\documentclass[11pt]{amsart}

\usepackage{amsmath,amsthm, amscd, amssymb, amsfonts}
\usepackage[all]{xy}
\usepackage{verbatim}
\usepackage{color}
\usepackage{array}
\usepackage{multirow}
\usepackage{hhline}
\usepackage{hyperref}

\usepackage{tikz}
\usetikzlibrary{patterns}
\usetikzlibrary{arrows}
\usetikzlibrary{shapes.misc, positioning}
\usetikzlibrary{backgrounds,scopes}
\usetikzlibrary{shapes.multipart,shapes.geometric}
\usetikzlibrary{decorations.pathmorphing}

\usepackage{enumitem}
\setlist[enumerate]{leftmargin=*,label=\rm{(\alph*)},topsep=4pt,itemsep=2pt,partopsep=2pt, parsep=2pt}

\newtheorem{theorem}{Theorem}[section]
\newtheorem{lemma}[theorem]{Lemma}
\newtheorem{proposition}[theorem]{Proposition}
\newtheorem{cor}[theorem]{Corollary}

\theoremstyle{definition}
\newtheorem{definition}[theorem]{Definition}

\theoremstyle{remark}
\newtheorem{remark}[theorem]{Remark}
\newcommand{\ydh}{{}^{H}_{H}\mathcal{YD}}

\newcommand{\nc}{\newcommand}
\nc{\D}{\Delta}
\nc{\cP}{\mathcal{P}}
\nc{\cU}{\mathcal{U}}
\nc{\cX}{\mathcal{X}}
\nc{\cE}{\mathcal{E}}
\nc{\cS}{\mathcal{S}}
\nc{\cA}{\mathcal{A}}
\nc{\cC}{\mathcal{C}}
\nc{\cO}{\mathcal{O}}
\nc{\cQ}{\mathcal{Q}}
\nc{\cB}{\mathcal{B}}
\nc{\cJ}{\mathcal{J}}
\nc{\cI}{\mathcal{I}}
\nc{\cM}{\mathcal{M}}
\nc{\cN}{\mathcal{N}}
\nc{\cF}{\mathcal{F}}
\nc{\cG}{\mathcal{G}}
\nc{\cL}{\mathcal{L}}
\nc{\cK}{\mathcal{K}}

\nc{\e}{\varepsilon}

\nc{\ba}{\mathbf{a}} \nc{\bb}{\mathbf{b}} \nc{\bt}{\mathbf{t}}
\nc{\br}{\mathbf{r}}
\nc{\bi}{\mathbf{i}} \nc{\bj}{\mathbf{j}}

\nc{\bB}{\mathbf{B}} \nc{\bS}{\mathbf{S}}  

\newcommand{\Sn}{{\mathbb S}}

\newcommand\id{\operatorname{id}}

\newcommand{\ydg}{{}^{\ku G}_{\ku G}\mathcal{YD}}

\newcommand\Ind{\operatorname{Ind}}

\newcommand\gr{\operatorname{gr}}

\newcommand\ad{\operatorname{ad}}

\newcommand\Alg{\operatorname{Alg}}

\def\ku{\Bbbk}

\def\ot{\otimes}

\def\N{\mathbb{N}}
\def\Z{\mathbb{Z}}
\def\F{\mathbb{F}}

\def\X{\mathbb{X}}

\def\B{\mathfrak{B}}

\def\fH{\mathfrak{H}}

\def\fX{\mathfrak{X}}
\def\fU{\mathfrak{U}}

\renewcommand{\_}[1]{_{\left( #1 \right)}}

\def\sl23{\mathbf{SL}(2,3)}

\def\ydsl23{{}^{\sl23}_{\sl23}\mathcal{YD}}

\begin{document}

%%%%%%%%% TITLE  AND OTHERS %%%%%%%%%%%%%%%%%%%%%

%%%%%%%%%%%%%%%%%%%%%%%%%%%%%%%%%%%%%%%%%%%%%%%%%%%%%%%%%%%

\title[]{Lifting in a non semisimple world}

\author[J. Arce, C. Vay ]{Jack Arce and Cristian Vay}

\address{\newline\noindent Pontificia Universidad Cat\'olica del Per\'u, Secci\'on Matem\'aticas, PUCP, Av. Universitaria 1801, San Miguel, Lima 32, Per\'u}
\email{jarcef@pucp.edu.pe}

\address{\newline\noindent Facultad de Matem\'atica, Astronom\'\i a, F\'\i sica y Computaci\'on,
	Universidad Nacional de C\'ordoba. CIEM--CONICET. Medina Allende s/n, 
	Ciudad Universitaria 5000 C\'ordoba, Argentina}
\email{cristian.vay@unc.edu.ar}


\begin{abstract}

This is a contribution to the problem of classifying all deformations -- a.~k.~a. liftings -- of the bosonization of a Nichols algebra $\B(V)$ over a cosemisimple and non-semisimple Hopf algebra $H$. Such a situation arises when the underlying field has positive characteristic or when $H$ is infinite-dimensional. Given an $H$-module $M$ that is an extension of $V$ by $H$, we first introduce an algebra $T(V)\#_MH$ which generalizes the usual bosonization $T(V)\#H$. Indeed, these two objects coincide when $M$ is a trivial extension. We provide necessary conditions for $T(V)\#_MH$ to be a Hopf algebra and a cocycle deformation of $T(V)\#H$. These conditions appear particularly natural when $H$ is a group algebra. We then prove that every lifting is a quotient of $T(V)\#_MH$ for some extension $M$. Echoing Archimedes, $T(V)\#_MH$ stands as a fulcrum over which we can pivot to lift the relations of the Nichols algebra. From this point, one can replicate the strategy proposed by Andruskiewitsch, Angiono, Garcia I., Masuoka and the second author to show that every lifting is a cocycle deformation of $\B(V)\#H$. We illustrate this idea with two examples of different nature. We classify all pointed liftings of the Fomin--Kirillov algebra $\cF\cK_3$ in characteristic $2$, and prove they are all cocycle deformations one another. We also prove that the Jordanian enveloping algebra of $\mathfrak{sl}(2)$ defined by Andruskiewitsch, Angiono and Heckenberger is a cocycle deformation of the bosonization of the Jordan plane over the infinite cyclic group.

\end{abstract}

\maketitle

\section{Introduction}

The Lifting Method was conceived by Andruskiewitsch and Schneider to classify Hopf algebras whose coradical is a (cosemisimple) Hopf subalgebra \cite{liftingmethod}. Large families of Hopf algebras have been classified by different authors following this method, mostly in characteristic zero and in the finite-dimensional case. Among these, the classification of the pointed Hopf algebras over abelian groups is the most outstanding \cite{annals,nicoivanagus,ivanagus}. One step in the Lifting Method consists in classifying all the deformations -- a.~k.~a. liftings -- of the bosonization of a Nichols algebra over the cosemisimple Hopf algebra. This work concerns this step. To fix ideas, let $H$ be a cosemisimple Hopf algebra,  $V$ a Yetter-Drinfeld module over $H$ and $\B(V)$ its Nichols algebra. A lifting of $\B(V)$ over $H$, (or a lifting of $\B(V)\#H$) is a Hopf algebra $L$ such that its coradical $L_0$ is a Hopf subalgebra isomorphic to $H$ and the associated graded Hopf algebra $\gr L=\oplus_{n>0}L_{n}/L_{n-1}$ is isomorphic to $\B(V)\#H$.

In most of the cases, it was proved that the liftings are cocycle deformations of $\B(V)\# H$, as for the finite-dimensional pointed Hopf algebras over abelian groups. Also, this property makes easier to verify whether a given Hopf algebra is indeed a lifting. Motivated by these facts, Andruskiewitsch, Angiono, Garcia I., Masuoka and the second author \cite{AAnMGV} proposed a strategy to compute all the liftings that are quotients of $T(V)\#H$ as cocycle deformations of $\B(V)\#H$, see Figure \ref{fig:lifting via cocycle semisimple}. On the other hand, Andruskiewitsch and the second author \cite{AV} showed that every lifting is a quotient of $T(V)\#H$ if $H$ is semisimple, using a powerful result by Ardizzoni, Menini and Stefan \cite{ardizzonimstefan}. Thus, one might obtain all the liftings of $\B(V)\#H$ as cocycle deformations when $H$ is semisimple -- as it is very well-known, $H$ is semisimple if it is finite-dimensional and the characteristic of the underlying field is zero.

\begin{figure}[h]
	\begin{tikzpicture}[scale=1,every node/.style={scale=1}]

		\node[align=center] (centrar) at (-5,0) {};
		
		\node[fill=white] (TVHl) at (-2.5,0) {$T(V)\#H$};
		
		\node[fill=white,label={above: {\tiny\it bigalois extensions}}] (TVH) at (0,0) {$T(V)\#H$};
		
		\node[fill=white] (TVHr) at (2.5,0) {$T(V)\#H$};
		
		\node[fill=white,minimum width=1cm] (L) at (-2.5,-3) {$L$};
		
		\node[fill=white,minimum width=1cm] (A) at (0,-3) {$A$};
		
		\node[fill=white] (BVH) at (2.5,-3) {$\B(V)\#H$};
		
		\node[align=center] (cd) at (5,0) {\tiny\it cocycle \\ \tiny\it deformations};
		
		\node[align=center] (cd) at (5,-3) {\tiny\it cocycle \\ \tiny\it deformations};
		
		\draw[->>] (TVHl) to (-2.5,-.9);
		\draw[dotted] (-2.5,-1) to (-2.5,-1.9);
		\draw[->>] (-2.5,-2) to (L);
		
		\draw[->>] (TVH) to (0,-.9);
		\draw[dotted] (0,-1) to (0,-1.9);
		\draw[->>] (0,-2) to (A);
		
		\draw[->>] (TVHr) to (2.5,-.9);
		\draw[dotted] (2.5,-1) to (2.5,-1.9);
		\draw[->>] (2.5,-2) to (BVH);
		
		%\draw[dashed] (-1.25,.5) to (-1.25,-3.5);
		%\draw[dashed] (1.25,.5) to (1.25,-3.5);

		\begin{scope}[on background layer]
			\draw[dashed] (-3.5,0) to (6,0);

			\draw[dashed] (-3.5,-3) to (6,-3);
		\end{scope}
	\end{tikzpicture}
	\caption{The strategy in \cite{AAnMGV} consists of constructing, step by step, a right Galois extension for the Hopf algebra on the right-hand side as a quotient of the one above following \cite{gunther}, starting from the algebras at the top. At each step, the Hopf algebra on the left-hand side is built from the other two following \cite{schauenburgL}. Thus, the algebras in the same row are cocycle deformation of one another. In the end, a lifting $L$ of $\B(V)\#H$ arises.}\label{fig:lifting via cocycle semisimple}
\end{figure}

A key point in deducing that every lifting is a quotient of $T(V)\#H$ when $H$ is semisimple lies in the fact that the extension
\begin{align*}
	0\longrightarrow L_0=H\longrightarrow L_1\longrightarrow L_1/L_0=V\#H\longrightarrow0
\end{align*}
splits as an extension of $H$-Hopf bimodules. However, this does not necessarily hold if $H$ is non-semisimple. In this work, we introduce and study a replacement for $T(V)\#H$. Let $M$ be an $H$-module that fits in a short exact sequence of $H$-modules
\begin{align*}
	0\longrightarrow H\longrightarrow M\longrightarrow V\longrightarrow0,
\end{align*}
where we consider $H$ as a module over itself with the adjoint action. Let us identify $V$ with a linear complement of $H$ in $M$. Our main object of study, the fulcrum, is the algebra  $T(V)\#_MH$ generated by $V$ and $H$ with the commutation rule $hv=(h\_{1}\cdot v)h\_{2}$ for all $v\in V$ and $h\in H$; on the right-hand side, we consider the action of $H$ on $V$ inside $M$. Of course, if $M$ splits, $T(V)\#_MH$ coincides with the usual bosonization $T(V)\#H$. Otherwise, it may happen that $hv\notin VH$, since $V$ is not a submodule of $M$.

In Section \ref{sec:TVMH}, we show that $T(V)\#_MH=T(V)\ot H$ as vector spaces using Bergman's Diamond Lemma \cite{diamondlemma}. Furthermore, we specify the conditions under which $T(V)\#_MH$ admits a Hopf algebra structure such that $H$ is a Hopf subalgebra and the comultiplication of $v\in V$ is given by $\Delta(v)=v\ot1+v\_{-1}\ot v\_{0}$. Let now $M'$ be another extension of $V$ by $H$. We provide sufficient conditions on $M'$ ensuring that $T(V)\#_{M'}H$ is a $(T(V)\#_MH,T(V)\#H)$-bigalois extension, which in turn implies that $T(V)\#_MH$ is a cocycle deformation of $T(V)\#H$ by \cite[Theorem 3.9]{schauenburgL}. The aforementioned conditions are quite natural and are almost automatically satisfied when $H$ is a group algebra or when $H$ is finite-dimensional.

Let $L$ be a lifting of $\B(V)\#H$. In Section \ref{sec:lifting non semisimple}, we demonstrate that there exists an $H$-module extension $M$ such that $L$ is a quotient of $T(V)\#_MH$. Moreover, if there exists an $H$-module extension $M'$ as in the previous paragraph, one can replicate the strategy proposed in \cite{AAnMGV} to show that every lifting is a cocycle deformation of $\B(V)\#H$ starting from the triple $T(V)\#_MH$, $T(V)\#_{M'}H$ and $T(V)\#H$, see Figure \ref{fig:lifting via cocycle nonsemisimple}.

We conclude this article by applying the above in two situation where the coradical is not semisimple, see Section \ref{sec:example}. First, we classify all pointed liftings of the Fomin-Kirillov algebra $\cF\cK_3$ in characteristic $2$, and we show that they are all cocycle deformations of one another. These are new examples of digital quantum groups \cite{digital}. We find here a phenomenon typical of positive characteristic, which can be observed in other works such as \cite{xiong}. Namely, in characteristic zero, the liftings are parameterized by the scalars that deform the defining relations of the Nichols algebra, and these scalars turn out to be algebraically independent. In contrast, in positive characteristic, the corresponding scalars are not necessarily algebraically independent. Moreover, they are subject to relations that also involve scalars coming from the module $M$. This phenomenon makes the analysis considerably more difficult, see Remark \ref{rmk:tetrahedron}. In the second situation, we deal with an infinite-dimensional coradical. We prove that the Jordanian enveloping algebra of $\mathfrak{sl}(2)$ defined by Andruskiewitsch, Angiono and Heckenberger in \cite[(1.4)]{corrigendum} is a cocycle deformation of the bosonization of the Jordan plane and the group of the integer numbers. 

\subsection*{Acknowledgments} 

J. A. was supported by PUCP-CAP 2025-PI1273. C. V. was partially supported by CONICET PIP 11220200102916CO, Foncyt PICT 2020-SERIEA-02847 and Secyt-UNC. C. V. carried out part of this work while in residence at the Institute for Computational and Experimental Research in Mathematics in Providence, RI, during the Categorification and Computation in Algebraic Combinatorics semester program, supported by the National Science Foundation under Grant No. DMS-1929284. C. V. thanks Nicol\'as Andruskiewitsch for teaching him to lift.

\begin{figure}[h]
	\begin{tikzpicture}[scale=1,every node/.style={scale=1}]

		\node[align=center] (centrar) at (-5,0) {};
		
		\node[fill=white] (TVHl) at (-2.5,0) {$T(V)\#_MH$};
		
		\node[fill=white,label={above: {\tiny\it bigalois extensions}}] (TVH) at (0,0) {$T(V)\#_{M'}H$};
		
		\node[fill=white] (TVHr) at (2.5,0) {$T(V)\#H$};
		
		\node[fill=white,minimum width=1cm] (L) at (-2.5,-3) {$L$};
		
		\node[fill=white,minimum width=1cm] (A) at (0,-3) {$A$};
		
		\node[fill=white] (BVH) at (2.5,-3) {$\B(V)\#H$};
		
		\node[align=center] (cd) at (5,0) {\tiny\it cocycle \\ \tiny\it deformations};
		
		\node[align=center] (cd) at (5,-3) {\tiny\it cocycle \\ \tiny\it deformations};
		
		\draw[->>] (TVHl) to (-2.5,-.9);
		\draw[dotted] (-2.5,-1) to (-2.5,-1.9);
		\draw[->>] (-2.5,-2) to (L);
		
		\draw[->>] (TVH) to (0,-.9);
		\draw[dotted] (0,-1) to (0,-1.9);
		\draw[->>] (0,-2) to (A);
		
		\draw[->>] (TVHr) to (2.5,-.9);
		\draw[dotted] (2.5,-1) to (2.5,-1.9);
		\draw[->>] (2.5,-2) to (BVH);
		
		%\draw[dashed] (-1.25,.5) to (-1.25,-3.5);
		%\draw[dashed] (1.25,.5) to (1.25,-3.5);

		\begin{scope}[on background layer]
			\draw[dashed] (-3.5,0) to (6,0);

			\draw[dashed] (-3.5,-3) to (6,-3);
		\end{scope}
	\end{tikzpicture}
	\caption{Adaptation of the strategy from \cite{AAnMGV} to the non-semisimple case.}\label{fig:lifting via cocycle nonsemisimple}
\end{figure}

\section{The fulcrum}\label{sec:TVMH}

Throughout this section,  $H$ is a Hopf algebra with bijective antipode over a field $\ku$, and we consider $H$ as a module over itself via the adjoint action, i.e., $h\cdot k=h\_{1}kS(h\_{2})$ for all $h,k\in H$. Furthermore, $V$ denotes an $H$-module and $M$ is an extension of $V$ by $H$, that is, there exists a short exact sequence of $H$-modules of the form
\begin{align*}
	0\longrightarrow H\longrightarrow M\longrightarrow V\longrightarrow0.
\end{align*}
We fix a set of linearly independent elements $\fX=\{x_i\}_{i\in I}$ such that $M=\ku\fX\oplus H$ as vector spaces. By abuse of notation, we denote in the same way the basis of $V$ determined by these elements. We also fix a basis $\fH$ of $H$ such that $1\in\fH$. Thus, there are bilinear maps $a_i,b_t:H\times \ku\fX\longrightarrow\ku$, $i\in I$ and $t\in\fH$, such that
\begin{align}\label{eq:M}
	h\cdot x&=\sum_{i\in I}a_i(h,x)x_i+\sum_{t\in\fH}b_t(h,x)t
\end{align}
for all $h\in H$ and $x\in\ku\fX$. Naturally, the action of $h$ on $x$ when viewed as an element in $V$ is
\begin{align}
	h\cdot_Vx= \sum_{i\in I}a_i(h,x)x_i.
\end{align}

We denote by $T_H(M)$ the quotient of the tensor algebra $T(M)$ by the relations of $H$. More precisely, this is the algebra generated by the elements $\fX$ and $\fH$ subject to the relations
\begin{align}
	\label{eq:H mult}
	hk&=\sum_{t\in\fH}m_t(h,k)t,
\end{align}
for $h,k\in\fH$, where $m_t:H\times H\longrightarrow\ku$ are the bilinear maps expressing the multiplication of $H$ in the basis $\fH$. Clearly, $H$ is a subalgebra of $T_H(M)$.

The fulcrum which we aim to study is the following algebra.

\begin{definition}\label{def:TV M H}
	The algebra $T(V)\#_MH$ is the quotient of $T_H(M)$ by the ideal $J_M$ defined by the commutation rule
	\begin{align*}
		hx=(h\_{1}\cdot x)h\_{2}
	\end{align*}
	for all $x\in M$ and $h\in H$. When $M=V\oplus H$ as $H$-module, we set $J=J_{V\oplus H}$ and $T(V)\#_M H$ reduces to the usual bosonization denoted $T(V)\#H$.
\end{definition}

We notice that the commutation rule in the definition can be written more explicitly using \eqref{eq:M}. Also, this commutation rule is only relevant when $x\notin H$ since otherwise we have $(h\_{1}\cdot x)h\_{2}=h\_{1}xS(h\_{2})h\_{3}=hx$ which is simply the multiplication in $H$.

We begin by highlighting the following properties of $T(V)\#_MH$.

\begin{proposition}\label{teo:TV M H}
	Let $\langle\fX\rangle$ be the subset of $T(V)\#_MH$ formed by the monomials (or words) in $\fX$. Then
	\begin{enumerate}
		\item\label{item:basis of the fulcrum} $\langle\fX\rangle\fH=\{Ah\mid A\in\langle\fX\rangle,\,h\in\fH\}$ is a basis $T(V)\#_MH$.
		\item\label{item:phi} The linear endomorphism $\phi$ of $T(V)\#_MH$ defined by $\phi(Ah)=hA$, $A\in\langle\fX\rangle$ and $h\in\fH$, is an isomorphism.
		\item\label{item:fulcrum graduado} $T(V)\#_MH$ is a graded vector space whose homogeneous component of degree $n\in\N$ is $(T(V)\#_MH)_n=\{Ah\mid A\in\langle\fX\rangle\,\mbox{of length $n$},\,h\in\fH\}$.
		\item\label{item:fulcrum subalgebras} $T(V)$ and $H$ are subalgebras.
	\end{enumerate}
\end{proposition}

\begin{proof}
We will apply Bergman's Diamond Lemma to prove \ref{item:basis of the fulcrum}. We first equip $\fH\cup\fX$ with an order such that $h>x$ for all $h\in\fH$ and $x\in\fX$, and extend this to the set of words in $\fH\cup\fX$ using the degree lexicographic order. Now, $T(V)\#_MH$ is defined by the reduction system $S$ consisting of the pairs
	\begin{align*}
		\left(hk,\,\sum_{t\in\fH}m_t(h,k)t\right)\,\mbox{ and }\,
		\left(hx,\,\sum_{i\in I}a_i(h\_{1},x)x_ih\_{2}+\sum_{t\in\fH}b_t(h\_{1},x)th\_{2}\right)
	\end{align*}
	for all $h,k\in\fH$ and $x\in\fX$. Clearly, the degree lexicographic order is compatible with $S$. The only possible ambiguities arise from monomials of the form $hkt$ and $hkx$, with $h,k,t\in\fH$ and $x\in\fX$. These are resolvable because of the associativity of the multiplication of $H$ and the associativity of the action of $H$ on  $M$. Since the $S$-irreducible monomials are precisely those in (a), these form a basis of $T(V)\#_MH$ by \cite[Theorem 1.2]{diamondlemma}, and this fact implies \ref{item:fulcrum graduado} and \ref{item:fulcrum subalgebras}.
	
	In order to prove (b), we observe that $T(V)\#_MH$ is an $H$-module algebra via the adjoint action. With this in mind, we can write in $T(V)\#_MH$
	\begin{align}\label{eq:h conmutando con una palabra}
		hx_{i_1}\cdots x_{i_n}=(h\_{1}\cdot x_{i_1})\cdots(h\_{n}\cdot x_{i_n})h\_{n+1}=h\_{1}\cdot(x_{i_1}\cdots x_{i_n})h\_{2},
	\end{align}
    for all $h\in H$ and $x_{i_1}, ..., x_{i_n}\in M$. Thus, $\phi(Ah)=hA=(h\_{1}\cdot A)h\_{2}$ is bijective with inverse $\phi^{-1}(Ah)=(S^{-1}(h\_{2})\cdot A)h\_{1}$.
\end{proof}

\begin{remark}
	Let $I$ be an ideal of $T(V)$ and set $R=T(V)/I$. Assuming that $I$ behaves well under the $H$-action, one can prove that $T(V)\#_MH/\langle I\rangle\simeq R\ot H$ as vector spaces. Thus, one could define $R\#_MH:=T(V)\#_MH/\langle I\rangle$ and thinks of it as a generalization of the Ore extensions. It would be interesting to investigate which properties of the Ore extensions can be extended to this more general construction.
\end{remark}

\subsection{A Hopf algebra structure for the fulcrum}\label{subsec:hopf for}

Here we also assume that $V$ is a Yetter-Drinfeld module over $H$ with coaction $\delta$. We endow $T_H(M)$ with the unique Hopf algebra structure such that $H$ is a Hopf subalgebra and the comultiplication of $v\in V$ is
\begin{align}\label{eq:comultiplication TVHM}
	\Delta(v)=v\ot1+\delta(v).
\end{align}

\begin{proposition}
	Under the above assumptions, $T(V)\#_MH$ is a Hopf algebra quotient of $T_H(M)$ if and only if
	\begin{align}
		\label{eq:hopf TVMH}
		\begin{split}
			\sum_{t\in\fH}b_t(h\_{1},x)t\_{1}h\_{2}\ot t\_{2}h\_{3}&=\\
			=\sum_{t\in\fH}b_t(h\_{1},x)th\_{2}&\ot h\_{3}+h\_{1}x\_{-1}\ot b_t(h\_{2},x\_{0})th\_{3}
		\end{split}
	\end{align}
	for all $h\in H$ and $x\in\ku\fX$. In such a case, the coradical of $T(V)\#_MH$ coincides with the coradical of $H$.
\end{proposition}

\begin{proof}
	The equivalence follows by comparing $\Delta(hx)$ and $\Delta((h\_{1}\cdot x)h\_{2})$. Now, if $T(V)\#_MH$ is a Hopf algebra, then $C_n=\oplus_{i=0}^n(T(V)\#_MH)_i$ is a coalgebra filtration. Therefore $C_0=H$ contains the coradical of $T(V)\#_MH$ by \cite[Proposition 11.1.1]{sweelder}.
\end{proof}

Let $M'=\ku\{y_i\}_{i\in I}\oplus H$ be an extension of $V$ by $H$ determined by the identification of $y_i$ with $x_i$ for all $i\in I$, and the action
\begin{align}\label{eq:M'}
	h\cdot y=h\cdot_Vy+\sum_{t\in\fH}b_t'(h,y)t
\end{align}
for all $h\in H$ and $y\in \ku\{y_i\}_{i\in I}$. Let $M''=\ku\{z_i\}_{i\in I}\oplus H$ be the direct sum of $V$ and $H$ where we identify $z_i$ with $x_i$ for all $i\in I$. We also transfer the coaction $\delta$ to $\ku\{y_i\}_{i\in I}$ and $\ku\{z_i\}_{i\in I}$ via the identifications above. Likewise, we transfer $b'_t$ to $H\times\ku\fX$. We adopt the same notation for the elements of $H$ as viewed in $T(V)\#_{M}H$, $T(V)\#_{M'}H$ or $T(V)\#H=T(V)\#_{M''}H$.

These three algebras are connected with each other, as stated in the following proposition, provided the equations
\begin{align}
	\label{eq:righ comodule}
	\sum_{t\in\fH}b'_t(h\_{1},x)t\_{1}h\_{2}\ot t\_{2}h\_{3}
	&=\sum_{t\in\fH}b'_t(h\_{1},x)th\_{2}\ot h\_{3}\quad\mbox{and}
	\\
	\noalign{\smallskip}
	\label{eq:left comodule}
	\begin{split}
		\sum_{t\in\fH}b'_t(h\_{1},x)t\_{1}h\_{2}\ot t\_{2}h\_{3}&=\\
		=\sum_{t\in\fH}b_t(h\_{1},x)th\_{2}&\ot h\_{3}+h\_{1}x\_{-1}\ot b'_t(h\_{2},x\_{0})th\_{3}.
	\end{split}
\end{align}
hold for all $h\in H$ and $x\in\ku\fX$.

\begin{lemma}\label{le:coactions}
	Assume \eqref{eq:righ comodule} and \eqref{eq:left comodule} hold for all $h\in H$ and $x\in\ku\fX$. Then $T(V)\#_{M'}H$ is a right (resp. left) comodule algebra over $T(V)\#H$ (resp. $T(V)\#_{M}H$). The
	right and left coactions are given by
	\begin{align*}
		\rho_r(h)&=h\ot h,\quad\rho_r(y)=y\ot1+\delta(z)\quad\mbox{and}
		\\
		\rho_l(h)&=h\ot h,\quad\rho_l(y)=x\ot1+\delta(y)
	\end{align*}
	for all $h\in H$ and $y\in\ku\{y_i\}_{i\in I}$.
\end{lemma}

\begin{proof}
	We observe that $T_H(M')=T_H(M'')=T_H(M)$ and hence we can consider $T_H(M')$ as a right and a left coalgebra over $T_H(M'')$ and $T_H(M)$ using the comultiplication. A straightforward computation using \eqref{eq:righ comodule} shows that $\Delta(J_{M'})$ is contained in $J_{M'}\ot T_H(M'')+T_H(M')\ot J_{M''}$. This implies that $\rho_r$ is an algebra map and a coaction. The analogous assertion for $\rho_l$ follows using  \eqref{eq:left comodule}.
\end{proof}

Further assumptions on $M$ and $M'$ are needed to obtain a bigalois extension. Consider the equations 
\begin{align}
	\label{eq:righ galois}
	&S(h\_{1})\ot
	\sum_{t\in\fH}b'_t(h\_{2},x)th\_{3}
	=\sum_{t\in\fH}b'_t(h\_{2},x)S(h\_{1})t\ot h\_{3}\quad\mbox{and}
	\\
	\noalign{\smallskip}
	\label{eq:left galois}
	\begin{split}
		&h\_{1}x\_{-2}\ot S(x\_{-1})\sum_{t\in\fH}b'_t(h\_{3},x\_{0})S(h\_{2})t=\\
		&=\sum_{t\in\fH}b'_t(h\_{1},x)th\_{2}\ot S(h\_{3})-\sum_{t\in\fH}b_t(h\_{1},x)t\_{1}h\_{2}\ot S(t\_{2}h\_{3})
	\end{split}
\end{align}
for $h\in H$ and $x\in\ku\fX$. Although equations \eqref{eq:righ comodule}-\eqref{eq:left galois} may look unmanageable, we will see in the next subsections they are almost immediate in some cases. 

\begin{theorem}\label{thm:fulcrum bigalois}
	If Equations \eqref{eq:righ comodule}, \eqref{eq:left comodule}, \eqref{eq:righ galois} and \eqref{eq:left galois} hold, then $T(V)\#_{M} H$ is a $(T(V)\#_MH,T(V)\#H)$-bigalois extension. Consequently, $T(V)\#_{M} H$ is a cocycle deformation of $T(V)\#H$ by \cite[Theorem 3.9]{schauenburgL}.
\end{theorem}

\begin{proof}
	We must prove that the corresponding Galois maps $\kappa_r$ and $\kappa_l$ are bijective, see for instance \cite[Section 2]{schauenburgL}. We will deduce this from the fact that $T_H(M')$ is a $(T_H(M'),T_H(M''))$-bigalois extension via the comultiplication. Indeed, let $\widetilde{\kappa}_r:T_H(M')\ot T_H(M')\longrightarrow T_H(M')\ot T_H(M'')$, $c\ot d\mapsto cd\_{1}\ot d\_{2}$, be the right Galois map for this extension. From the definition of $\rho_r$, we see that the bijectivity of $\kappa_r$ can be deduced by showing that
	\begin{align*}
		\widetilde{\kappa}_r\bigl(J_{M'}\ot T_H(M')+T_H(M')\ot J_{M'}\bigr)=J_{M'}\ot T_H(M'')+T_H(M')\ot J_{M''}.
	\end{align*}
	Clearly, $\widetilde{\kappa}_r(J_{M'}\ot T_H(M'))\subset J_{M'}\ot T_H(M'')$, and $\widetilde{\kappa}_r(T_H(M')\ot J_{M'})\subset J_{M'}\ot T_H(M'')+T_H(M')\ot J_{M''}$ by \eqref{eq:righ comodule}. This proves the left-hand side is contained in the right-hand side. On the other hand, the inverse of $\widetilde{\kappa}_r$ is given by $c\ot d\mapsto cS(d\_{1})\ot d\_{2}$. Then, $\widetilde{\kappa}_r^{-1}(J_{M'}\ot T_H(M''))\subset J_{M'}\ot T_H(M')$, and a straightforward computation using \eqref{eq:righ galois} shows that $\widetilde{\kappa}_r^{-1}(T_H(M')\ot J_{M''})\subset J_{M'}\ot T_H(M')+T_H(M')\ot J_{M'}$. This proves the reverse inclusion, and hence $\kappa_r$ is bijective. The bijectivity of $\kappa_l$ can be proved similarly using \eqref{eq:left comodule} and \eqref{eq:left galois}. 
\end{proof}

\subsubsection{The pointed case}\label{subsubsec:the pointed case}

The conditions \eqref{eq:hopf TVMH}, \eqref{eq:righ comodule} and \eqref{eq:left comodule} are easy to verify when $H=\ku G$ is a group algebra. Indeed, we can take $\fH=G$ and assume that $\delta(x_i)=g_i\ot x_i$ for certain $g_i\in G$, $i\in I$. Hence \eqref{eq:hopf TVMH} holds if and only if
\begin{align}\label{eq:M pointed case}
	g\cdot x_i=g\cdot_V x_i+b(g,x_i)(1-g\cdot g_i),
\end{align}
with $b=b_1$, while \eqref{eq:left comodule} and \eqref{eq:righ comodule} hold if and only if
\begin{align}\label{eq:M' pointed case}
	g\cdot y_i=g\cdot_V y_i+b(g,x_i),
\end{align}
for all $g\in G$ and $i\in I$. Moreover, \eqref{eq:M pointed case} makes $M$ an extension of $\ku G$-module if and only if \eqref{eq:M' pointed case} does the same for $M'$, and this equivalent to
\begin{align}\label{eq:M extension pointed case}
	b(gh,x_i)=b(g,h\cdot_Vx_i)+b(h,x_i)
\end{align}
for all $g,h\in G$ and $i\in I$. We observe that these extensions satisfy \eqref{eq:righ galois} and \eqref{eq:left galois} immediately. Summarizing, we have the following.

\begin{cor}\label{cor:TVMH pointed case}
	The following statements are equivalent:
	\begin{enumerate}
		\item $M$ is an extension of $V$ by $\ku G$ with action as in \eqref{eq:M pointed case}.
		\item Equation \eqref{eq:M extension pointed case} holds.
		\item $T(V)\#_M\ku G$ is a pointed Hopf algebra quotient of $T_H(M)$.
	\end{enumerate}
	If these statements hold, then $T(V)\#_{M'}\ku G$ is a $(T(V)\#_M\ku G,T(V)\#\ku G)$-bigalois extension, and $T(V)\#_M\ku G$ is a cocycle deformation of $T(V)\#\ku G$. \qed
\end{cor}

\subsubsection{The finite-dimensional case}

In this case Equations \eqref{eq:righ galois} and \eqref{eq:left galois} are not necessary.

\begin{theorem}\label{thm:fulcrum bigalois fin dim}
	If $H$ is finite-dimensional and Equations \eqref{eq:righ comodule} and \eqref{eq:left comodule} hold, then $T(V)\#_{M} H$ is a $(T(V)\#_MH,T(V)\#H)$-bigalois extension. Consequently, $T(V)\#_{M} H$ is a cocycle deformation of $T(V)\#H$.
\end{theorem}

\begin{proof}
	As for Theorem \ref{thm:fulcrum bigalois}, we must prove that $\kappa_r$ and $\kappa_l$ are bijective. However, it is enough to verify their injectivity as $H$ is finite-dimensional. Indeed, the image of the subspace
	\begin{align*}
		\bigoplus_{s+t\leq n}(T(V)\#_{M'}H)_s\ot (T(V)\#_{M'}H)_{t}
	\end{align*}
	under $\kappa_r$
	is contained in  $\bigoplus_{s+t\leq n}(T(V)\#_{M'}H)_s\ot (T(V)\#H)_{t}$ for all $n\in\N$, and these are finite-dimensional vector spaces of the same dimension by Theorem \ref{teo:TV M H}. In order to prove that $\kappa_r$ is injective, we proceed as follows. Let $p=\sum_{A,B}A h_{A,B}\ot B k_{A,B}$ where the sum runs over all words $A,B$ in $\{y_i\}_{i\in I}$, here $h_{A,B},k_{A,B}\in H$. If $p\neq0$, we pick words $A_0$ and $B_0$ with maximal lengths among the words with $h_{A,B}\neq0\neq k_{A,B}$. Let $s$ and $t$ be the lengths of $A_0$ and $B_0$, respectively. Then
	\begin{align*}
		(\pi_{s+t}\ot\pi_0)\kappa_r(p)&=\sum_{\substack{\operatorname{length}A=s\\\operatorname{length}B=t}}\pi_{s+t}(A h_{A,B} B k_{A,B}{}\_{1})\ot k_{A,B}{}\_{2}\\
		&=\sum_{\substack{\operatorname{length}A=s\\\operatorname{length}B=t}}A\,\pi_{t}(h_{A,B} B)\,k_{A,B}{}\_{1}\ot k_{A,B}{}\_{2}
	\end{align*}
	where, by an abuse of notation, we use $\pi_n$ to denote the projections from $T(V)\#_{M'}H$ and $T(V)\#H$ onto their homogeneous components for any $n\in\N$. Now, if $\ker\kappa_r(p)=0$, the above equality implies that
	\begin{align*}
		0=&\sum_{\operatorname{length}B=t}\,\pi_{t}(h_{A_0,B} B )\,k_{A_0,B}{}\_{1}\ot k_{A_0,B}{}\_{2}\\
		=&(\pi_{t}\phi\ot\id_H)\kappa'\left(\sum_{\operatorname{length}B=t}Bh_{A_0,B}\ot k_{A_0,B}\right)
	\end{align*}
	where $\kappa'$ denotes the automorphism $q\ot h\mapsto qh\_{1}\ot h\_{2}$ of $(T(V)\#_{M'}H)_{t}\ot H$ and $\phi$ is the automorphism of $T(V)\#_{M'}H$ given by Theorem \ref{teo:TV M H} (b). We notice that $\phi(T(V)\#_{M'}H)_{t}\subseteq\oplus_{\ell\leq t}(T(V)\#_{M'}H)_\ell$ by definition of $\phi$, see also \eqref{eq:h conmutando con una palabra}, and the restriction of $\phi$ induces an automorphism of $\oplus_{\ell\leq t}(T(V)\#_{M'}H)_\ell$. Since $H$ is finite-dimensional, we deduce that $\pi_{t}\phi$ must be injective in $(T(V)\#_{M'}H)_{t}$, and therefore $0=\sum_{\operatorname{length}B=t}Bh_{A_0,B}\ot k_{A_0,B}$. In particular, $h_{A_0,B_0}\ot k_{A_0,B_0}$ must be $0$, which contradicts the fact that $h_{A_0,B_0}\neq0\neq k_{A_0,B_0}$. In conclusion $\kappa_r(p)\neq0$ if $p\neq0$, and hence $\kappa_r$ is injective. The proof for $\kappa_l$ is analogous.
\end{proof}

\section{Liftings}\label{sec:lifting non semisimple}

Let $H$ be a cosemisimple Hopf algebra with bijective antipode over a field $\ku$, and $\B(V)$ the Nichols algebra of a Yetter-Drinfeld module $V\in\ydh$. The following characterization of the liftings extends \cite[Proposition 2.4]{AV} to the non-semisimple case.

\begin{theorem}\label{teo:liftings}
	A Hopf algebra $L$ is a lifting of $\B(V)$ over $H$ if and only if there exists an extension $M$ of $V$ by $H$ and an epimorphism of Hopf algebras $\phi:T(V)\#_MH\longrightarrow L$ such that $\phi_{|H\oplus V\#H}:H\oplus V\#H\longrightarrow L_1$ is a bijection.
\end{theorem}

\begin{proof}
	Let us assume that $L$ is a lifting. Thus, \cite[Corollary 2.18]{ardizzonimstefan} guarantees the existence of a right $H$-module coalgebra map
	$\Pi:L\longrightarrow H$
	such that $\Pi_{|H}=\id_{H}$. Let $\delta_l=(\Pi\ot\id)\Delta$ and $\delta_r=(\id\ot\Pi)\Delta$ be the induced $H$-coactions on $L$. By \cite[Theorem 3.71]{ardizzonimstefan}, $L_n\simeq L_n^{co\,\delta_r}\#H$ as right $H$-module coalgebras for all $n\in\N$. On the other hand, by \cite[Lemma 1.1]{AN}, $L_1=H\oplus P_1$ as $H$-bicomodules with
	\begin{align*}
		P_1=L_1\cap\ker\Pi=\{x\in L\mid\Delta(x)=\delta_r(x)+\delta_l(x)\}
	\end{align*}
	and hence the projection  $\pi:L_1\longrightarrow P_1$ with $\ker\pi=H$ is a morphism of right $H$-modules and $H$-bicomodules.
	We deduce then that $V\#H\simeq L_1/L_0\simeq P_1\simeq L_1^{co\,\delta_r}\#H$, and $\pi_{|P_1^{co\,\delta_r}}:P_1^{co\,\delta_r}\longrightarrow V$ is a bijection. Since $V$ is invariant by the adjoint action of $H$ in $V\#H$, it follows that $\ad H(P_1^{co\,\delta_r})\subset P_1^{co\,\delta_r}\oplus\ker\pi=P_1^{co\,\delta_r}\oplus H$. Let us identify $P_1^{co\,\delta_r}$ with $V$ using $\pi$. Therefore $M=V\oplus H$ is an extension of $H$-modules of $V$ by $H$ satisfying the first assumptions in \S\ref{subsec:hopf for} with $\delta=\delta_l$. Thus, there exists a Hopf algebra map $\phi:T(V)\#_MH\longrightarrow L$ such that $\phi_{|V\oplus H}=\id_{|V\oplus H}$ which is surjective as $L$ is generated by $L_1=H\oplus V\#H$.
	
	The reciprocal statement follows as in \cite[Proposition 2.4]{AV}.
\end{proof}

\subsection{Pointed liftings}\label{subsec:lifting pointed}
Assume that $H=\ku G$ is a group algebra. As in \S\ref{subsubsec:the pointed case}, we may choose a basis $\{x_i\}_{i\in I}$ of $V$ such that $\delta(x_i)=g_i\ot x_i$ for certain $g_i\in G$, $i\in I$. In this case, the action on the module $M$ involved in Theorem \ref{teo:liftings} is given by \eqref{eq:M pointed case}, in view of Corollary \ref{cor:TVMH pointed case}. Hence, we obtain the following.

\begin{cor}\label{cor:pointed liftings}
	A pointed Hopf algebra $L$ is a lifting of $\B(V)$ over $\ku G$, if and only if $L$ is generated by $\{g\}_{g\in G}$ and $\{x_i\}_{i\in I}$ satisfying that
	\begin{enumerate}
		\item the comultiplication is given by
		\begin{align*}
			\Delta(g)=g\ot g\quad\mbox{and}\quad\Delta(x_i)=x_i\ot1+g_i\ot x_i,
		\end{align*}
		\item there exists a bilinear map $b:\ku G\times V\longrightarrow\ku$ such that
		\begin{align*}
			g x_ig^{-1}=g\cdot_V x_i+b(g,x_i)(1-gg_ig^{-1})
		\end{align*}
		for all $g\in G$ and $i\in I$, and
		\item $L_1=\ku\{x_i\}_{i\in I}\,\ku G\oplus\ku G$.
	\end{enumerate}
	\qed
\end{cor}

\section{Examples}\label{sec:example}

\subsection{An example in positive characteristic}

Let $\X$ be the rack of transpositions in $\Sn_3$, or equivalently, the affine rack $(\Z_3,2)$. This is the cyclic group $\Z_3$ endowed with the operation $i\rhd j=2i-j$, for $i,j\in\Z_3$. The enveloping group $G_{\X}$ of $\X$ \cite{EG,J} is defined as
\begin{align}\label{eq:enveloping group}
	G_{\X}=\langle g_0,\,g_1,\,g_2\mid g_ig_j=g_{i\rhd j}g_i,\, i,j\in\Z_3\rangle.
\end{align}
From now on, we fix a non-abelian quotient group $G$ of $G_{\X}$. We still denote $g_i$ the image in $G$ of the generators of $G_{\X}$. Then $g_0,$ $g_1$ and $g_2$ are pairwise different elements of $G$, they generate $G$ and form a conjugacy class which naturally identifies with $\X$ as a rack. Using this identification, we let $G$ acts on $\Z_3$ and write $g\cdot i$ for the action of $g\in G$ on $i\in\Z_3$; for instance, $g_j\cdot i=j\rhd i$ for  $j\in\Z_3$. The centralizer $G^{g_0}$ of $g_0$ is the cyclic group generated by $g_0$, see for instance \cite[Section 5.1]{GHV}.

Throughout this section $\ku$ is a field of characteristic $2$. Let $(\varepsilon,\ku)$ be the one-dimensional representation of $G^{g_0}$ given by $\varepsilon:G^{g_0}\longrightarrow\ku^*$, $\varepsilon(g_0)=1$; in characteristic $2$, this coincides with the sign representation. The structure of the Yetter-Drinfeld module
\begin{align*}
	V=M(g_0,\varepsilon)=\Ind_{G^{g_0}}^{G}\ku=\ku G\ot_{\ku G^{g_0}}\ku\in\ydg
\end{align*}
is determined by
\begin{align}\label{eq:tetrahedron YD module}
	g\cdot x_i=x_{g\cdot i}\quad\mbox{and}\quad\delta(x_i)=g_i\ot x_i
\end{align}
for all $g\in G$ and $i\in\Z_3$, where
\begin{align*}
	x_0=e\ot1,\quad
	x_1=g_2\cdot x_0=g_2\ot1\quad\mbox{and}\quad
	x_2=g_1\cdot x_0=g_1\ot1.
\end{align*}

The Nichols algebra $\B(V)$ of $V$ identifies with the Fomin--Kirillov algebra $\cF\cK_3$ \cite{fominkirillov}. This is the quotient of the tensor algebra $T(V)$ by the ideal generated by the elements
\begin{align*}
	R_{i,j}=x_ix_j+x_{i\rhd j}x_i+x_{(i\rhd j)\rhd i}x_{i\rhd j}
\end{align*}
for $i,j\in\Z_3$. We notice that $R_{i,i}=x_i^2$, $R_{i,j}=R_{i\rhd j,i}=R_{j,i\rhd j}$ and  $(i\rhd j)\rhd i=j$ for all $i,j\in\Z_3$. A basis of $\B(V)$ consists of all possible words $m_1m_2m_3$ where $m_i$ is an element in the $i$-th row of the list:
\begin{align*}
	&1,\,x_0,\\
	&1,\,x_1,\,x_1x_0,\\
	&1,\,x_2.
\end{align*}

In order to define the liftings of $\B(V)$ over $\ku G$, we need to introduce some notation. Recall
\S\ref{subsubsec:the pointed case}. Let $\cL$ be a collection of bilinear maps $\lambda:\ku G\times V\longrightarrow\ku$ classifying (up to isomorphism) all the extensions of $V$ by $\ku G$. For $\lambda\in\cL$, we denote $M_\lambda=V\oplus\ku G$ the extension determined by 
\begin{align*}
	g\cdot x_i=x_{g\cdot i}+\lambda(g,x_i)(1+g_{g\cdot i}),
\end{align*}
$g\in G$ and $i\in\Z_3$, cf. \eqref{eq:M pointed case}.  We set $T_\lambda=T(V)\#_{M_\lambda}\ku G$ for the associated Hopf algebra given by Definition \ref{def:TV M H}. For brevity, we write $\lambda_{i,j}=\lambda(g_i,x_j)$ and then
\begin{align}\label{eq:lij}
\lambda_{i,j\rhd k}+\lambda_{j,k}=\lambda_{i\rhd j,i\rhd k}+\lambda_{i,k}
\end{align}
for all $i,j,k\in\Z_3$ thanks to \eqref{eq:M extension pointed case} and the relation $g_ig_j=g_{i\rhd j}g_i$. Let $\cM_\lambda$ be the set of matrices $\mu\in\ku^{\Z_3\times\Z_3}$ satisfying the equations
\begin{align}\label{eq:mij}
	&\mu_{i,j}=\mu_{i\rhd j,i}=\mu_{j,i\rhd j}\quad\mbox{and}\\
	\label{eq:rel linda}
	\begin{split}
		&\mu_{i,j}+\mu_{k\rhd i,k\rhd j}=\\
		&=\lambda_{k,i}(\lambda_{k,i\rhd j}+\lambda_{i,j})+\lambda_{k,j}(\lambda_{k,i}+\lambda_{j,(i\rhd j)})+\lambda_{k,(i\rhd j)}(\lambda_{k,j}+\lambda_{(i\rhd j),i})
	\end{split}
\end{align}
for all $i,j\in\Z_3$. Finally, we introduce the elements
\begin{align*}
	r_{i,j}= \lambda_{i,j}x_i+\lambda_{i\rhd j,i} x_{i\rhd j}+\lambda_{j,i\rhd j}x_{j}\in T_\lambda,
\end{align*}
for $i,j\in\Z_3$; notice that $r_{i,j}=r_{i\rhd j,i}=r_{j,i\rhd j}$ and $r_{i,i}=\lambda_{i,i}x_i$.

\begin{definition}
	Given $\lambda\in\cL$ and $\mu\in\cM_\lambda$, we denote $L_{\lambda\mu}$ the quotient of $T_\lambda$ by the ideal generated by the elements
	\begin{align*}
		R_{i,j}+r_{i,j}+\mu_{i,j}(1+g_ig_j)
	\end{align*}
	for $i,j\in\Z_3$.
\end{definition}

As in \eqref{eq:M' pointed case}, we let $M'_\lambda$, $\lambda\in\cL$, be the extension of $V$ by $\ku G$ determined by
\begin{align*}
	g\cdot y_i=y_{g\cdot i}+\lambda(g,y_i),
\end{align*}
$g\in G$ and $i\in\Z_3$. Then the algebra $T'_\lambda=T(V)\#_{M'_\lambda}\ku G$ is a $(T_\lambda,T(V)\#\ku G)$-bigalois extension by Corollary \ref{cor:TVMH pointed case}. Let $R'_{i,j}$ and $r'_{i,j}$ be the elements in $T_\lambda'$ obtained from $R_{i,j}$ and $r_{i,j}$ by replacing the letters $x_i$ with $y_i$.

\begin{definition}
	Given $\lambda\in\cL$ and $\mu\in\cM_\lambda$, we denote $A_{\lambda\mu}$ the quotient of $T'_\lambda$ by the ideal generated by the elements
	\begin{align*}
		R_{i,j}'+r'_{i,j}+\mu_{i,j}
	\end{align*}
	for $i,j\in\Z_3$.
\end{definition}

We may define the above algebras for any $\mu\in\ku^{\Z_3\times\Z_3}$ but this is pointless in view of the following results.

\begin{lemma}\label{le:A neq 0}
\begin{enumerate}
	\item\label{item:A neq 0} If $\lambda\in\cL$ and $\mu\in\cM_\lambda$, then $A_{\lambda\mu}\neq0$.
	\item\label{item:L} If $\mu\in\ku^{\Z_3\times\Z_3}$ and $L_{\lambda\mu}\neq0$, then $\mu\in\cM_\lambda$.
\end{enumerate}	
\end{lemma}

\begin{proof}
	\ref{item:A neq 0} We think of $A_{\lambda\mu}$ as an algebra presented by generators and relations, and apply Bergman's Diamond Lemma. That is, we have to verify that the ambiguities $g_ky_iy_j$, $y_3y_3y_1$, $y_3y_3y_2$, $y_3y_1y_1$ and $y_3y_2y_2$ are resolvable. For that purpose, we carried out long and tedious calculations with the help of \cite{GAP,GBNP}, in which we have to use the relations satisfied by $\lambda$ and $\mu$. It turn out that all of these ambiguities are resolvable except for $y_3y_1y_1$ and $y_3y_2y_2$, which produce the new relation $y_2y_1y_2= y_1y_2y_1 + (\lambda_{2,1} + \lambda_{1,2})y_2y_1 + ( \lambda_{2,1}+ \lambda_{1,2})y_1y_2 + (\lambda_{1,2}\lambda_{2,1} +\mu_{3,3}+\mu_{1,1}+\mu_{1,2})y_2 +(\lambda_{1,2}\lambda_{2,1} +\mu_{3,3}+\mu_{2,2}+\mu_{1,2})y_1 + \lambda_{2,1}(\mu_{2,2}+\mu_{1,2}) + \lambda_{1,2}(\mu_{1,1}+\mu_{1,2})$. This gives rise to new ambiguities to analyze, namely $g_iy_2y_1y_2$, $y_3y_2y_1y_2$, $y_2y_2y_1y_2$ and $y_2y_1y_2y_2$, which also turn out to be resolvable.
	
	\ref{item:L} Reasoning as above, it is enough to note that $g_kx_ix_j$ is resolvable only if \eqref{eq:rel linda} holds. Also, the equalities satisfied by the elements $R_{i,j}$ and $r_{i,j}$ imply \eqref{eq:mij}.
\end{proof}

We are now ready to state the main result of this subsection.

\begin{theorem}\label{teo:liftings of FK}
	Let $\lambda\in\cL$ and $\mu\in\cM_\lambda$. Then
	\begin{enumerate}
		\item\label{item:Llmn}
		$L_{\lambda\mu}$ is a Hopf algebra quotient of $T_\lambda$ which is a lifting of $\B(V)\#\ku G$, and every such lifting is isomorphic to one of this form.
		\item\label{item:A} $A_{\lambda\mu}$ is a $(L_{\lambda\mu},\B(V)\#\ku G)$-bigalois extension.
		\item\label{item:cocycle} $L_{\lambda\mu}$ is a cocycle deformation of $\B(V)\#\ku G$.
		\item\label{item:iso} $L_{\lambda\mu}\simeq L_{\tilde\lambda\tilde\mu}$ if and only if there exist a rack automorphism $\varphi$ of $\X$ and $\phi,\phi^0,\phi^1,\phi^2\in\ku$ such that $\phi\neq0$,
		\begin{align*}
			\tilde\lambda_{\varphi(i),\varphi(j)}&=\lambda_{i,j}+\phi^{i\rhd j}+\phi^j,\\
			\tilde\mu_{\varphi(i),\varphi(j)}&=\mu_{i,j}+\phi^i\phi^j+\phi^{i\rhd j}\phi^i+\phi^j\phi^{i\rhd j}\quad\mbox{and}\\
			(\phi+1)\tilde\lambda_{\varphi(i),\varphi(j)}&=\phi^i\tilde\lambda_{\varphi(i),\varphi(j)}+\phi^{i\rhd j}\tilde\lambda_{\varphi(i\rhd j),\varphi(i)}+\phi^{j}\tilde\lambda_{\varphi(j),\varphi(i\rhd j)}=0
		\end{align*}
		for all $i,j\in\Z_3$. In such a case, the assignment $g_i\mapsto\tilde g_{\varphi(i)}$ and $x_i\mapsto \phi\tilde x_{\varphi(i)}+\phi^i(1+\tilde g_{\varphi(i)})$, $i\in\Z_3$, induces a Hopf algebra isomorphism between $L_{\lambda\mu}$ and $L_{\tilde\lambda\tilde\mu}$; here we write the generators of $L_{\tilde\lambda\tilde\mu}$ with a tilde.
	\end{enumerate}
\end{theorem}

\begin{proof}
	In order to prove \ref{item:Llmn}, we observe that in $T_\lambda$ it holds that
	\begin{align*}
		\Delta(R_{i,j}+r_{i,j})=(R_{i,j}+r_{i,j})\ot1+ g_ig_j\ot(R_{i,j}+r_{i,j}),
	\end{align*}
	for all $i,j\in\Z_3$. Therefore $L_{\lambda\mu}$ is a Hopf algebra quotient of $T_\lambda$. Also, $(L_{\lambda\mu})_1=V\ku G\oplus\ku G$ by \ref{item:cocycle}. Then $L_{\lambda\mu}$ is a lifting by Corollary \ref{cor:pointed liftings}. Reciprocally, if $L$ is such a lifting, it is a quotient of $T_\lambda$ by Corollary \ref{cor:pointed liftings}, and one deduces that it is isomorphic to some $L_{\lambda\mu}$ using the above formula as usual, see for instance \cite[Section 4]{AAnMGV}.
	
	We now address \ref{item:A} using the strategy from \cite[Section 5]{AAnMGV}, see Figure \ref{fig:lifting tetrahedron}. We follow it as a recipe, pointing out the key ingredients in our case. For a more detailed exposition of how the strategy works we refer to \cite[Section 5]{AAnMGV}, \cite{tour} or \cite{GVafin}; in particular, the latter considers, among others,  the lifting of the Fomin--Kirillov algebra $\cF\cK_3$ in characteristic zero.

	Recall $T'_\lambda$ is a right Galois extension of $T(V)\#\ku G$ by Corollary \ref{cor:TVMH pointed case}. Let $\gamma_1:T(V)\#\ku G\longrightarrow T'_\lambda$ be the corresponding section. One can see that $\gamma_1(R_{i,j})=R'_{i,j}$ proceeding as in \cite[Section 4.5]{tour}. Let $Y$ be the subalgebra of $T(V)\#\ku G$ generated by $S(R_{i,j})=g_j^{-1}g_{i}^{-1}R_{i,j}$, $i,j\in\Z_3$. This is isomorphic to the tensor algebra $T(Y)$ by \cite[Lemma 28]{AV} with $U$ the linear span of $R_{i,j}$'s. Hence the assignment $S(R_{i,j})=g_j^{-1}g_{i}^{-1}R_{i,j}\mapsto g_j^{-1}g_{i}^{-1}(R'_{i,j}+r'_{i,j}+\mu_{i,j})$ determines $\varphi\in\Alg^{\B(V)\#\ku G}(Y,T'_\lambda)$ for any $\mu\in\cM$. By Lemma \ref{le:A neq 0}, $T'_\lambda/\langle\varphi(Y^ +)\rangle=A_{\lambda\mu}\neq0$ and hence $A_{\lambda\mu}$ is a right Galois extension of $\B(V)\#\ku G$ by \cite[Theorem 8]{gunther}. Then, we deduce that $A_{\lambda\mu}$ is also a left Galois extension of $L_{\lambda\mu}$ using \cite[Corollary 5.12]{AAnMGV}; here we have to use the formula of $\gamma_1$. This completes the proof of \ref{item:A}, and \ref{item:A} implies \ref{item:cocycle}, cf. \cite{schauenburgL}.
	
	The proof of \ref{item:iso} is standard. We outline the main ideas and leave the details to the readers. If $\Phi:L_{\lambda\mu}\longrightarrow L_{\tilde\lambda\tilde\mu}$ is an isomorphism of Hopf algebras, then $\Phi(g_i)=\tilde{g}_{\varphi(i)}$ and $\Phi(x_i)=\phi \tilde{x}_{\varphi(i)}+\phi^i(1+\tilde{g}_{\varphi(i)})$ for all $i\in\Z_3$. The equality $\Phi(g_ix_j)=\Phi(g_i)\Phi(x_j)$ implies that $\tilde{\lambda}_{\phi(i),\phi(j)}=\lambda_{i,j}+\phi^i+\phi^{i\rhd j}$ for all $i,j\in\Z_3$. The remaining identities in the statement follow from the fact that $\Phi(R_{i,j}+r_{i,j}+\mu_{i,j}(1+g_ig_j))$ must be $0$.	
\end{proof}
	
	\begin{figure}[h]
		\begin{tikzpicture}[scale=1,every node/.style={scale=1}]

			\node[align=center] (centrar) at (-5,0) {};
			
			\node[fill=white] (THM) at (-2.5,1.5) {$T_{\ku G}(M)$};
			
			\node[fill=white,label={above: {\tiny\it bigalois extensions}}] (THMprima) at (0,1.5) {$T_{\ku G}(M')$};
			
			\node[fill=white] (THMsegunda) at (2.5,1.5) {$T_{\ku G}(M'')$};

			\node[fill=white] (TVHl) at (-2.5,0) {$T_\lambda$};
			
			\node[fill=white] (TVH) at (0,0) {$T'_\lambda$};
			
			\node[fill=white] (TVHr) at (2.5,0) {$T(V)\#\ku G$};
			
			\node[fill=white,minimum width=1cm] (L2) at (-2.5,-1.5) {$L_{\lambda\mu}$};
			
			\node[fill=white,minimum width=1cm] (A2) at (0,-1.5) {$A_{\lambda\mu}$};
			
			\node[fill=white] (BVH2) at (2.5,-1.5) {$\B(V)\#\ku G$};
			
			\node[align=center] (cd) at (5,1.5) {\tiny\it cocycle \\ \tiny\it deformations};
			
			\node[align=center] (cd) at (5,0) {\tiny\it cocycle \\ \tiny\it deformations};
			
			\node[align=center] (cd) at (5,-1.5) {\tiny\it cocycle \\ \tiny\it deformations};
			
			\draw[->>] (THM) to (TVHl);
			
			\draw[->>] (THMprima) to (TVH);
			
			\draw[->>] (THMsegunda) to (TVHr);
			
			\draw[->] (THMsegunda) to node[midway, above] {\tiny $\gamma_0=\id$} (THMprima);

			\draw[->>] (TVHl) to (L2);
			
			\draw[->>] (TVH) to (A2);
			
			\draw[->] (TVHr) to node[midway, above] {\tiny $\gamma_1$} (TVH);
			
			\draw[->>] (TVHr) to (BVH2);
			
			\draw[dashed] (BVH2) (A2);

			\begin{scope}[on background layer]
				\draw[dashed] (-3.5,1.5) to (6,1.5);
				
				\draw[dashed] (-3.5,0) to (6,0);
				
				\draw[dashed] (-3.5,-1.5) to (6,-1.5);
				
			\end{scope}
		\end{tikzpicture}
		\caption{The strategy from \cite{AAnMGV} applied to the liftings of the Fomin--Kirillov algebra $\cF\cK_3$ in characteristic 2.}\label{fig:lifting tetrahedron}
	\end{figure}

\subsubsection{The binary case}

Here we let $\ku=\F_2$ be the field of two elements.
We use \emph{Mathematica} \cite{mathematica} to compute all the pairs of matrices $\lambda\in\cL$ and $\mu\in\cM_\lambda$ for $G=G_\X$ over $\F_2$; for the enveloping group, $\lambda\in\cL$ if and only if it satisfies \eqref{eq:lij}. There are $32$ of such pairs. Using Theorem \ref{teo:liftings of FK} \ref{item:iso}, and with the help of GAP \cite{GAP}, we conclude that up to isomorphism there are $10$ liftings of the Fomin--Kirillov algebra $\cF\cK_3$ over $\F_2G_\X$. We list them in Table \ref{table:soluciones}. We notice that the first two classes of liftings are quotients of the bosonization $T(V)\#\F_2G_\X$, whereas we need non-trivial fulcrums to construct the remaining liftings.

Assume now $G=\Sn_3\simeq G_\X/\langle g_0^2\rangle$, cf. \cite[Section 5.1]{GHV}. Then $\lambda\in\cL$ if and only if it satisfies \eqref{eq:lij} and $\lambda_{i,j}=\lambda_{i,i\rhd j}$ for all $i,j\in\Z_3$ as $g_i^2=g_0^2$ in $G_\X$ \cite[Lemma 2.18]{GHV}. We observe that at least one $\lambda$ of each isomorphism class in Table \ref{table:soluciones} satisfies these equalities. Then Table \ref{table:soluciones} also describes the isomorphism classes of the liftings of $\B(V)\#\F_2\Sn_3$. These are new examples of finite digital quantum groups \cite{digital}.

\begin{table}\label{table:soluciones}
	\centering
	\resizebox{.9\textwidth}{!}{
	\renewcommand{\arraystretch}{1.2}
	\begin{minipage}{0.39\textwidth}
		\centering
		\begin{tabular}{| c | c |}
			\hline
			$\lambda$ & $\mu$
			\\
			\hline
			\rule{0pt}{1cm}
			$ \begin{bmatrix} 0& 0& 0 \\ 0& 0& 0 \\ 0& 0& 0 \end{bmatrix} $ &
			$ \begin{bmatrix} 0& 0& 0 \\ 0& 0& 0 \\ 0& 0& 0 \end{bmatrix} $
			\\[1.5em]
			\rule{0pt}{1cm}
			$ \begin{bmatrix} 0& 0& 0 \\ 0& 0& 0 \\ 0& 0& 0 \end{bmatrix} $ &
			$ \begin{bmatrix} 1& 1& 1 \\ 1& 1& 1 \\ 1& 1& 1 \end{bmatrix} $
			\\[1.5em]
			\rule{0pt}{1cm}
			$ \begin{bmatrix} 0& 0& 0 \\ 1& 0& 1 \\ 1& 1& 0 \end{bmatrix} $ &
			$ \begin{bmatrix} 0& 1& 1 \\ 1& 1& 1 \\ 1& 1& 1 \end{bmatrix} $
			\\[1.5em]
			\rule{0pt}{1cm}
			$ \begin{bmatrix} 0& 0& 0 \\ 1& 0& 1 \\ 1& 1& 0 \end{bmatrix} $ &
			$ \begin{bmatrix} 1& 0& 0 \\ 0& 0& 0 \\ 0& 0& 0 \end{bmatrix} $
			\\[1.5em]
			\rule{0pt}{1cm}
			$ \begin{bmatrix} 0& 1& 1 \\ 0& 0& 0 \\ 1& 1& 0 \end{bmatrix} $ &
			$ \begin{bmatrix} 0& 0& 0 \\ 0& 1& 0 \\ 0& 0& 0 \end{bmatrix} $
			\\[1.5em]
			\rule{0pt}{1cm}
			$ \begin{bmatrix} 0& 1& 1 \\ 0& 0& 0 \\ 1& 1& 0 \end{bmatrix} $ &
			$ \begin{bmatrix} 1& 1& 1 \\ 1& 0& 1 \\ 1& 1& 1 \end{bmatrix} $
			\\[1.5em]
			\rule{0pt}{1cm}
			$ \begin{bmatrix} 0& 1& 1 \\ 1& 0& 1 \\ 0& 0& 0 \end{bmatrix} $ &
			$ \begin{bmatrix} 0& 0& 0 \\ 0& 0& 0 \\ 0& 0& 1 \end{bmatrix} $
			\\[1.5em]
			\rule{0pt}{1cm}
			$ \begin{bmatrix} 0& 1& 1 \\ 1& 0& 1 \\ 0& 0& 0 \end{bmatrix} $ &
			$ \begin{bmatrix} 1& 1& 1 \\ 1& 1& 1 \\ 1& 1& 0 \end{bmatrix} $
			\\[1.5em]
			\hline
		\end{tabular}
	\end{minipage}
	\hfill
	\begin{minipage}{0.39\textwidth}
		\centering
		\begin{tabular}{| c | c |}
			\hline
			$\lambda$ & $\mu$
			\\
			\hline
			\rule{0pt}{1cm}
			$ \begin{bmatrix} 0& 0& 0 \\ 0& 0& 0 \\ 0& 0& 0 \end{bmatrix} $ &
			$ \begin{bmatrix} 0& 1& 1 \\ 1& 0& 1 \\ 1& 1& 0 \end{bmatrix} $
			\\[1.5em]
			\rule{0pt}{1cm}
			$ \begin{bmatrix} 0& 0& 0 \\ 0& 0& 0 \\ 0& 0& 0 \end{bmatrix} $ &
			$ \begin{bmatrix} 1& 0& 0 \\ 0& 1& 0 \\ 0& 0& 1 \end{bmatrix} $
			\\[1.5em]
			\rule{0pt}{1cm}
			$ \begin{bmatrix} 0& 0& 0 \\ 1& 0& 1 \\ 1& 1& 0 \end{bmatrix} $ &
			$ \begin{bmatrix} 0& 0& 0 \\ 0& 1& 0 \\ 0& 0& 1 \end{bmatrix} $
			\\[1.5em]
			\rule{0pt}{1cm}
			$ \begin{bmatrix} 0& 0& 0 \\ 1& 0& 1 \\ 1& 1& 0 \end{bmatrix} $ &
			$ \begin{bmatrix} 1& 1& 1 \\ 1& 0& 1 \\ 1& 1& 0 \end{bmatrix} $
			\\[1.5em]
			\rule{0pt}{1cm}
			$ \begin{bmatrix} 0& 1& 1 \\ 0& 0& 0 \\ 1& 1& 0 \end{bmatrix} $ &
			$ \begin{bmatrix} 0& 1& 1 \\ 1& 1& 1 \\ 1& 1& 0 \end{bmatrix} $
			\\[1.5em]
			\rule{0pt}{1cm}
			$ \begin{bmatrix} 0& 1& 1 \\ 0& 0& 0 \\ 1& 1& 0 \end{bmatrix} $ &
			$ \begin{bmatrix} 1& 0& 0 \\ 0& 0& 0 \\ 0& 0& 1 \end{bmatrix} $
			\\[1.5em]
			\rule{0pt}{1cm}
			$ \begin{bmatrix} 0& 1& 1 \\ 1& 0& 1 \\ 0& 0& 1 \end{bmatrix} $ &
			$ \begin{bmatrix} 0& 1& 1 \\ 1& 0& 1 \\ 1& 1& 1 \end{bmatrix} $
			\\[1.5em]
			\rule{0pt}{1cm}
			$ \begin{bmatrix} 0& 1& 1 \\ 1& 0& 1 \\ 0& 0& 1 \end{bmatrix} $ &
			$ \begin{bmatrix} 1& 0& 0 \\ 0& 1& 0 \\ 0& 0& 0 \end{bmatrix} $
			\\[1.5em]
			\hline
		\end{tabular}
	\end{minipage}
	\hfill
	\begin{minipage}{0.39\textwidth}
		\centering
		\begin{tabular}{| c | c |}
			\hline
			$\lambda$ & $\mu$
			\\
			\hline
			\rule{0pt}{1cm}
			$ \begin{bmatrix} 1& 0& 0 \\ 0& 1& 0 \\ 1& 1& 1 \end{bmatrix} $ &
			$ \begin{bmatrix} 0& 0& 0 \\ 0& 0& 0 \\ 0& 0& 0 \end{bmatrix} $
			\\[1.5em]
			\rule{0pt}{1cm}
			$ \begin{bmatrix} 1& 0& 0 \\ 1& 1& 1 \\ 0& 0& 1 \end{bmatrix} $ &
			$ \begin{bmatrix} 0& 0& 0 \\ 0& 0& 0 \\ 0& 0& 0 \end{bmatrix} $
			\\[1.5em]
			\rule{0pt}{1cm}
			$ \begin{bmatrix} 1& 1& 1 \\ 0& 1& 0 \\ 0& 0& 1 \end{bmatrix} $ &
			$ \begin{bmatrix} 0& 0& 0 \\ 0& 0& 0 \\ 0& 0& 0 \end{bmatrix} $
			\\[1.5em]
			\hline
			\rule{0pt}{1cm}
			$ \begin{bmatrix} 1& 0& 0 \\ 0& 1& 0 \\ 1& 1& 1 \end{bmatrix} $ &
			$ \begin{bmatrix} 0& 1& 1 \\ 1& 0& 1 \\ 1& 1& 0 \end{bmatrix} $
			\\[1.5em]
			\rule{0pt}{1cm}
			$ \begin{bmatrix} 1& 0& 0 \\ 1& 1& 1 \\ 0& 0& 1 \end{bmatrix} $ &
			$ \begin{bmatrix} 0& 1& 1 \\ 1& 0& 1 \\ 1& 1& 0 \end{bmatrix} $
			\\[1.5em]
			\rule{0pt}{1cm}
			$ \begin{bmatrix} 1& 1& 1 \\ 0& 1& 0 \\ 0& 0& 1 \end{bmatrix} $ &
			$ \begin{bmatrix} 0& 1& 1 \\ 1& 0& 1 \\ 1& 1& 0 \end{bmatrix} $
			\\[1.5em]
			\hline
			\rule{0pt}{1cm}
			$ \begin{bmatrix} 1& 1& 1 \\ 1& 1& 1 \\ 1& 1& 1 \end{bmatrix} $ &
			$ \begin{bmatrix} 0& 0& 0 \\ 0& 0& 0 \\ 0& 0& 0 \end{bmatrix} $
			\\[1.5em]
			\hline
			\rule{0pt}{1cm}
			$ \begin{bmatrix} 1& 1& 1 \\ 1& 1& 1 \\ 1& 1& 1 \end{bmatrix} $ &
			$ \begin{bmatrix} 0& 1& 1 \\ 1& 0& 1 \\ 1& 1& 0 \end{bmatrix} $
			\\[1.5em]
			\hline
		\end{tabular}
	\end{minipage}
\hfill
\begin{minipage}{0.39\textwidth}
	\centering
	\begin{tabular}{| c | c |}
		\hline
		$\lambda$ & $\mu$
		\\
		\hline
		\rule{0pt}{1cm}
		$ \begin{bmatrix} 1& 0& 0 \\ 0& 1& 0 \\ 1& 1& 1 \end{bmatrix} $ &
		$ \begin{bmatrix} 1& 0& 0 \\ 0& 1& 0 \\ 0& 0& 1 \end{bmatrix} $
		\\[1.5em]
		\rule{0pt}{1cm}
		$ \begin{bmatrix} 1& 0& 0 \\ 1& 1& 1 \\ 0& 0& 1 \end{bmatrix} $ &
		$ \begin{bmatrix} 1& 0& 0 \\ 0& 1& 0 \\ 0& 0& 1 \end{bmatrix} $
		\\[1.5em]
		\rule{0pt}{1cm}
		$ \begin{bmatrix} 1& 1& 1 \\ 0& 1& 0 \\ 0& 0& 1 \end{bmatrix} $ &
		$ \begin{bmatrix} 1& 0& 0 \\ 0& 1& 0 \\ 0& 0& 1 \end{bmatrix} $
		\\[1.5em]
		\hline
		\rule{0pt}{1cm}
		$ \begin{bmatrix} 1& 0& 0 \\ 0& 1& 0 \\ 1& 1& 1 \end{bmatrix} $ &
		$ \begin{bmatrix} 1& 1& 1 \\ 1& 1& 1 \\ 1& 1& 1 \end{bmatrix} $
		\\[1.5em]
		\rule{0pt}{1cm}
		$ \begin{bmatrix} 1& 0& 0 \\ 1& 1& 1 \\ 0& 0& 1 \end{bmatrix} $ &
		$ \begin{bmatrix} 1& 1& 1 \\ 1& 1& 1 \\ 1& 1& 1 \end{bmatrix} $
		\\[1.5em]
		\rule{0pt}{1cm}
		$ \begin{bmatrix} 1& 1& 1 \\ 0& 1& 0 \\ 0& 0& 1 \end{bmatrix} $ &
		$ \begin{bmatrix} 1& 1& 1 \\ 1& 1& 1 \\ 1& 1& 1 \end{bmatrix} $
		\\[1.5em]
		\hline
		\rule{0pt}{1cm}
		$ \begin{bmatrix} 1& 1& 1 \\ 1& 1& 1 \\ 1& 1& 1 \end{bmatrix} $ &
		$ \begin{bmatrix} 1& 0& 0 \\ 0& 1& 0 \\ 0& 0& 1 \end{bmatrix} $
		\\[1.5em]
		\hline
		\rule{0pt}{1cm}
		$ \begin{bmatrix} 1& 1& 1 \\ 1& 1& 1 \\ 1& 1& 1 \end{bmatrix} $ &
		$ \begin{bmatrix} 1& 1& 1 \\ 1& 1& 1 \\ 1& 1& 1 \end{bmatrix} $
		\\[1.5em]
		\hline
	\end{tabular}
	\end{minipage}
}
	\vspace{1em}
	\caption{These are all the pairs $(\lambda,\mu)$ that give rise to  a lifting of $\cF\cK_3$ over $\F_2G_\X$. Pairs yielding isomorphic liftings are displayed together.}
\end{table}

\begin{remark}\label{rmk:tetrahedron}
	We also have considered to classify the liftings of the tetrahedron Nichols algebra in characteristic $2$ discovered by Gra\~na, Heckenberger and Vendramin \cite{GHV}. This algebra is presented by four generators subject to quadratic relations similar to those of the Fomin--Kirillov algebra, together a cubic relation. However, the computations required to characterize the parameter deforming the cubic relation are quite involved, and we need more time to complete them. We plan to address this in a future work on the classification of pointed liftings over the tetrahedron group.
\end{remark}

\subsection{An example with infinite-dimensional coradical}
Here $\ku$ is a field of characteristic zero. 
Let $G=\langle g\rangle$ be the free abelian group generated by $g$ and $V=\ku\{x_1,x_2\}\in\ydg$ the Yetter-Drinfeld module given by
\begin{align*}
	g\cdot x_1=x_1,\quad g\cdot x_2=x_2+x_1,\quad\delta(x_1)=g\ot x_1\quad\mbox{and}\quad\delta(x_2)=g\ot x_2.
\end{align*}
The Nichols algebra of $V$ is the so-called Jordan plane \cite{AAHjordan} and it is defined by the relation $x_1x_2-x_2x_1-\frac{1}{2}x_1^2$.

Let $M$ be the extension of $V$ by $\ku G$ given by
\begin{align*}
	g\cdot x_1=x_1+(1-g)\quad\mbox{and}\quad g\cdot x_2=x_2+x_1.
\end{align*}
Thus, the Jordanian enveloping algebra of $\mathfrak{sl}(2)$ defined by Andruskiewitsch, Angiono and Heckenberger in \cite[(1.4)]{corrigendum} can be presented as
\begin{align*}
	\fU^{\mathrm{jordan}}=T(V)\#_M\ku G/\left\langle\mbox{$x_1x_2-x_2x_1-\frac{1}{2}x_1^2+x_2+\frac{1}{2}x_1$}\right\rangle.
\end{align*}
In \cite{corrigendum} the authors showed that $\fU^{\mathrm{jordan}}$ is a lifting of $\B(V)$ over $\ku G$, and realized it as a Hopf subalgebra of  a Hopf algebra introduced by Ohn \cite{ohn}.

Let $M'$ be the extension of $V$ by $\ku G$ given by
\begin{align*}
	g\cdot y_1=y_1+1\quad\mbox{and}\quad g\cdot y_2=y_2+y_1.
\end{align*}
We set $\fU'=T(V)\#_{M'}\ku G/\langle x_1x_2-x_2x_1-\frac{1}{2}x_1^2+x_2+\frac{1}{2}x_1\rangle$. Similar to $\fU^{\rm jordan}$, $\fU'$ can be described as the iterated Ore extension $\ku G[x_1;\id,\delta_1][x_2;\sigma,\delta_2]$ with $\delta_1(g)=-g$, while $\sigma$ and $\delta_2$ are as in \cite[Proposition 6]{corrigendum}.

It is straightforward to verify that the coactions in Lemma \ref{le:coactions} induce corresponding coactions in $\fU'$ over $\fU^{\rm jordan}$ and $\B(V)\#\ku G$ making it a left and right comodule algebra, respectively. Moreover, using the strategy proposed in \cite[Section 5]{AAnMGV} as in Theorem \ref{teo:liftings of FK}, we obtain  the following.

\begin{proposition}
	The algebra $\fU'$ is a $(\fU^{\rm jordan},\B(V)\#\ku G)$-bigalois extension and therefore $\fU^{\rm jordan}$ is a cocycle deformation of $\B(V)\#\ku G$. \qed
\end{proposition}

\bibliographystyle{amsplain}

\end{document}